\documentclass[11pt]{amsart}

\usepackage{microtype}
\usepackage{amsmath,amssymb,amsthm,mathtools}
\usepackage{enumitem}
\usepackage{aliascnt}
\usepackage[hidelinks]{hyperref}
\usepackage[nameinlink,noabbrev]{cleveref}
\usepackage[edges]{forest}

\theoremstyle{plain}
\newtheorem{theorem}{Theorem}[section]

\newaliascnt{proposition}{theorem}
\newtheorem{proposition}[proposition]{Proposition}
\aliascntresetthe{proposition}

\newaliascnt{lemma}{theorem}
\newtheorem{lemma}[lemma]{Lemma}
\aliascntresetthe{lemma}

\newaliascnt{corollary}{theorem}
\newtheorem{corollary}[corollary]{Corollary}
\aliascntresetthe{corollary}

\newaliascnt{question}{theorem}
\newtheorem{question}[question]{Question}
\aliascntresetthe{question}

\theoremstyle{definition}
\newaliascnt{definition}{theorem}
\newtheorem{definition}[definition]{Definition}
\aliascntresetthe{definition}

\newaliascnt{example}{theorem}
\newtheorem{example}[example]{Example}
\aliascntresetthe{example}

\theoremstyle{remark}
\newaliascnt{remark}{theorem}
\newtheorem{remark}[remark]{Remark}
\aliascntresetthe{remark}

\newcommand{\Hh}{\mathcal H}
\newcommand{\NN}{\mathbb N}
\newcommand{\Ninf}{\NN_{-\infty}}
\newcommand{\htree}{\operatorname{ht}}
\newcommand{\botvec}{\boldsymbol{-\infty}}
\newcommand{\eval}[1]{\left[\!\left[#1\right]\!\right]}
\newcommand{\defi}[1]{\textbf{#1}}
\newcommand{\Kraft}{\mathcal K}
\newcommand{\Trees}{\mathcal T}
\newcommand{\Languages}{\mathcal L}
\newcommand{\Variety}{\mathsf V}
\newcommand{\wideNinf}{\widehat{\Ninf}}
\newcommand{\wideHh}{\widehat{\Hh}}

\title{The Tropical Algebra of Binary-Tree Height}
\author{Grant Molnar}
\date{}


\crefname{theorem}{theorem}{theorems}
\crefname{proposition}{proposition}{propositions}
\crefname{lemma}{lemma}{lemmas}
\crefname{corollary}{corollary}{corollaries}
\crefname{question}{question}{questions}
\crefname{definition}{definition}{definitions}
\crefname{example}{example}{examples}
\crefname{remark}{remark}{remarks}
\Crefname{theorem}{Theorem}{Theorems}
\Crefname{proposition}{Proposition}{Propositions}
\Crefname{lemma}{Lemma}{Lemmas}
\Crefname{corollary}{Corollary}{Corollaries}
\Crefname{question}{Question}{Questions}
\Crefname{definition}{Definition}{Definitions}
\Crefname{example}{Example}{Examples}
\Crefname{remark}{Remark}{Remarks}

\begin{document}

\begin{abstract}
The binary-tree height recursion defines an algebra
$\mathcal{H}$ on $\mathbb{N}\cup\{-\infty\}$, with join given by
$\max$ and product
\[
a\star b=\max\{a,b\}+1.
\]
We show that weighted evaluation of a labelled tree depends only on the
greatest depth of each label. Single-tree profiles are exactly the vectors
satisfying the binary Kraft inequality, while finite joins realize every vector in
$\bigl(\mathbb{N}\cup\{-\infty\}\bigr)^n$; hence the $n$-variable term
operations form the free algebra $\mathcal{H}^n$. We also classify $\mathcal{H}$'s
compatible semilattice operation, subalgebras, endomorphisms, congruences,
and finite quotients, and recover the dyadic-composition spectrum at the
full-linear boundary.
\end{abstract}

\maketitle
\section{Introduction}\label{sec:introduction}

A \defi{finite full rooted binary tree} is a finite rooted tree in which
every vertex has either zero or two children.  The \defi{height} of such a tree is the greatest graph distance from the root to a leaf, so the one-vertex tree has height \(0\).  If \(T\) is not
a leaf, we let \(T_1\) and \(T_2\) be the rooted subtrees whose roots are the two
children of the root of \(T\).  Then
\[
  \htree(T)=\max\{\htree(T_1),\htree(T_2)\}+1.
\]
Thus height is computed from the leaves upward by taking a maximum and adding
one at each grafting.

The same recursive rule extends to arbitrary leaf weights.  We let \(\NN\coloneqq\{0,1,2,\ldots\}\) and \(\Ninf\coloneqq\NN\cup\{-\infty\}\), where \(-\infty<n\) and \(-\infty+n\coloneqq-\infty\) for \(n\in\NN\).  Abstracting the two subtree values to \(a,b\in\Ninf\) gives the binary operation
\[
  a\star b=\max\{a,b\}+1.
\]

\begin{definition}[Height algebra]\label{def:height-algebra}
The \defi{height algebra} is
\[
  \Hh\coloneqq(\Ninf,\vee,\star,-\infty),
\]
where
\begin{equation}\label{eq:height-operations}
  a\vee b\coloneqq\max\{a,b\},
  \qquad\text{and}\qquad
  a\star b\coloneqq\max\{a,b\}+1.
\end{equation}
\end{definition}

The element \(-\infty\) is least for \(\vee\), but it is not absorbing
for \(\star\).

\begin{samepage}
We fix \(n\geq1\), and put \(X\coloneqq\{x_1,\dots,x_n\}\).  All rooted binary
trees below are finite and full.  Unless stated otherwise, they are nonplane,
so the two children of an internal vertex are unordered.  Labelled trees are
considered up to rooted isomorphism preserving all leaf labels.  The depth of
a leaf is its graph distance from the root.

\begin{definition}[Labelled trees and tree monomials]
\label{def:tree-monomial}
A \defi{labelled tree over \(X\)} is a rooted binary tree in which each leaf
carries either an element of \(X\) or the \defi{silent label} \(-\infty\).
A leaf carrying \(-\infty\) is a \defi{silent leaf}; every other leaf is
\defi{active}.  A tree whose leaves are all silent is an \defi{all-silent
tree}.  The \defi{grafting} \(S\circ T\) is obtained by adjoining a new root
whose children are the roots of \(S\) and \(T\).  A \defi{tree monomial} is a
term built from \(X\cup\{-\infty\}\) using \(\star\) alone.  Choosing a
planar embedding of a labelled tree produces a tree monomial, and different
choices differ only by interchanging sibling subterms.  Thus a nonplane
labelled tree records a tree monomial modulo commutativity of \(\star\).
\end{definition}
\end{samepage}

To evaluate a labelled tree over \(X\), we assign the weight \(u_i\in\Ninf\) to
each leaf carrying \(x_i\) and the value \(-\infty\) to each silent leaf.
We combine the two branch values with \(\star\) from the leaves to the root.  The
product \(\star\) is commutative but not associative, so the value can depend
on the rooted bracketing even though the order of two siblings does not
matter.  What information about the labelled tree survives this evaluation?

For each label, evaluation remembers only the greatest depth at which that label occurs.  The tree value is the maximum, over the labels that occur, of the assigned weight plus this greatest depth.
We call the corresponding coefficient vector the \defi{depth profile}; it is
the key invariant developed below.

Two questions organize the paper.

The first question is this: when do two labelled trees, or finite joins of labelled
trees, induce the same operation on \(\Hh\)?  For a single tree, the weighted
depth formula shows that evaluation depends only on its depth profile
(\Cref{thm:weighted-depth,cor:tree-operation-equality}).  The profiles
realized by single trees are exactly the vectors satisfying the binary Kraft
inequality (\Cref{thm:kraft-profile}), whereas finite joins realize every
vector in \(\Ninf^n\) (\Cref{thm:all-profiles}).  The resulting term
operations are max-plus linear in the standard tropical sense
(\Cref{thm:tropical-normal-form}); see also
\cite[Chapter~1, \S1.1]{MaclaganSturmfels2015}.  The classical full-linear
problem, in which each variable occurs exactly once, is recovered by
restricting the same depth-profile classification
(\Cref{thm:linear-spectra}).

We recall that an \(n\)-variable \defi{term operation} is obtained by evaluating
a term in \(\vee\), \(\star\), and \(-\infty\).  The \defi{variety generated
by \(\Hh\)}, denoted by \(\Variety(\Hh)\), is the class of algebras satisfying
every identity valid in \(\Hh\).

\begin{theorem}[Tree evaluation and normal forms]\label{thm:main-normal-forms}
We let \(\Hh\) be the height algebra of \Cref{def:height-algebra}.
\begin{enumerate}[label=\textnormal{(\arabic*)},leftmargin=*]
  \item Every labelled rooted binary tree induces a max-plus linear operation.
  Its \(i\)th coefficient is the greatest depth at which \(x_i\) occurs,
  with coefficient \(-\infty\) when \(x_i\) does not occur.  Two labelled
  trees induce the same operation if and only if these coefficient vectors
  agree.

  \item A vector \(d=(d_1,\dots,d_n)\in\Ninf^n\) occurs as the coefficient
  vector of a single labelled tree if and only if
  \[
    \sum_{i=1}^n2^{-d_i}\leq1,
  \]
  where \(2^{-\infty}\) is interpreted as \(0\).  Such a vector has a
  realization with an internal root if and only if none of its coordinates
  is zero.

  \item Every vector in \(\Ninf^n\) is the coefficient vector of a finite
  join of labelled trees.  The algebra of \(n\)-variable term operations of \(\Hh\)
  is naturally isomorphic to \(\Hh^n\).  This is the free algebra on \(n\)
  generators in \(\Variety(\Hh)\).  Under finite-language semantics, the
  empty language supplies an absorbing zero, while a unital completion
  requires a distinct algebraic unit \(\varepsilon\), interpreted
  heuristically as an empty-tree state.

  \item Among terms in which each variable occurs exactly once, distinct
  bracketings in a fixed variable order induce distinct operations.  If
  variable permutations are also allowed, the number of induced operations
  is the number of ordered representations of \(1\) as a sum of \(n\)
  negative powers of \(2\).
\end{enumerate}
\end{theorem}

The second question is this: how much of the algebraic structure of the height algebra is forced by
the height product itself?  The answer separates what belongs to the
algebra from what arises through the tree presentation
(\Cref{thm:unique-join,thm:subalgebras,thm:endomorphisms,thm:congruences,cor:finite-quotients,cor:residual-finiteness}).

We recall that an algebra is \defi{rigid} if its only automorphism is the
identity map.  It is \defi{residually finite} if any two distinct elements
remain distinct in some finite quotient.

\begin{theorem}[Intrinsic structure]\label{thm:main-intrinsic}
For the height algebra \(\Hh\), the following statements hold.
\begin{enumerate}[label=\textnormal{(\arabic*)},leftmargin=*]
  \item Maximum is the unique commutative, associative, idempotent operation
  on \(\Ninf\) with identity \(-\infty\) over which \(\star\) distributes.

  \item The subalgebras of \(\Hh\) are \(\{-\infty\}\) and the upward
  tails \(\{-\infty\}\cup\{m,m+1,m+2,\dots\}\) for \(m\in\NN\).

  \item Every endomorphism of \(\Hh\) is either the constant map to
  \(-\infty\) or a translation of the finite heights.  In particular,
  \(\Hh\) is rigid.

  \item Every proper nontrivial congruence collapses a unique final tail of
  finite heights and fixes every other element.  Consequently, the
  nontrivial finite quotients of \(\Hh\) are obtained by capping all
  heights above a finite threshold.

  \item The algebra \(\Hh\), and every finite power of \(\Hh\), is
  residually finite.
\end{enumerate}
\end{theorem}

Finite-language semantics also clarify the lower boundaries of height algebra theory.  The value
\(-\infty\), the absence of a tree, and a formal grafting unit obey different
laws.  The latter two states are introduced only when the semantics requires
them, and all three are compared in \Cref{ex:three-boundary-states}.

The weighted-depth formula drives the main semantic development of the paper.
It reduces the classification of induced operations to the classification of
coefficient vectors.  \Cref{sec:kraft} identifies those arising from single
trees, while \Cref{sec:normal-forms,sec:tree-languages} show that finite joins
remove the Kraft constraint and give the resulting quotient its language
semantics.  The same profile classification yields the classical full-linear
boundary in \Cref{sec:linear-spectra}.  Separately, the uniqueness of the
compatible join shows that the semilattice structure is intrinsic to the
height product and supports the structural classifications of
\Cref{sec:structure}.  The final section records further questions about finite
axiomatizability and higher branching arity.

\subsection*{Prior work}

The height recursion belongs naturally to the theory of bottom-up tree
evaluation.  Ganardi, Hucke, Lohrey, and Noeth use
\(f^I(a_1,\dots,a_r)=1+\max_i a_i\) for grammar-compressed trees
\cite[Theorem~16]{GanardiEtAl2018}, while bottom-up evaluation in algebras is
part of the standard framework of tree automata
\cite[\S\S1.3 and 2.1]{GecsegSteinby2015}.  Weighted tree automata place the
same kind of recursion in run and initial-algebra semantics, with height as an
explicit max-plus example
\cite[Preface and Example~3.2.4]{FulopVogler2024}.  Recent work also studies
the generating power of initial-algebra semantics
\cite[Introduction]{DrosteEtAl2024}.  These theories provide the ambient
semantics for the present paper.  We fix the numerical height algebra and
determine the quotient that its evaluation induces on repeated labels and
finite joins.

The distributive interaction between multiplication and a semilattice is
studied more generally in the theory of semilattice-ordered algebras
\cite[Definition~1.1]{PilitowskaZamojska2014}, including extensions by zero
and unit constants \cite[Definitions~2.1--2.6]{PilitowskaZamojska2020}.  For
the particular multiplication \(a\star b=\max\{a,b\}+1\), the compatible
semilattice operation is itself determined by the product (\Cref{thm:unique-join}).  This rigidity
reduces the remaining intrinsic structure to explicit classifications of
subalgebras, endomorphisms, congruences, and finite quotients.

Other algebras of rooted binary trees retain substantially more of the tree
than height does.  The Loday--Ronco Hopf algebra
\cite[\S3]{LodayRonco1998} and Loday's dendriform arithmetic of planar binary
trees \cite[Introduction and \S\S5--6]{Loday2002} operate directly on tree
shapes, whereas the present paper passes to a numerical quotient.  At the
full-linear boundary, that quotient retains the labelled leaf-depth tuple
from associative-spectrum theory
\cite[\S2.3 and Proposition~4.2.2]{HuangLehtonen2023}.  Its enumeration is the
dyadic-composition sequence studied asymptotically by Krenn and Wagner
\cite[Equation~(1.4) and Theorem~II]{KrennWagner2016}.  Allowing repeated
labels and finite joins extends this classical boundary to the full profile
quotient developed below.

The adjective ``tropical'' refers to the max-plus linear normal forms, in the
max-plus convention dual to the usual min-plus convention
\cite[Chapter~1, \S1.1]{MaclaganSturmfels2015}.  The height product is not the
multiplication of the tropical semiring.  Its nonassociativity records rooted
bracketing through leaf depth.  This mechanism also differs from
nonassociative tropical operations arising from entropy deformations
\cite[\S2.1, especially Equation~(2.1)]{MarcolliTedeschi2015} and from
idempotent nonassociative extensions of max-times geometry
\cite[Introduction and \S1.1]{Briec2025}.  The point of contact is therefore
the max-plus form of the resulting term operations, not a shared
multiplicative structure.

\section*{Acknowledgment}

The author thanks Chris Nowlin for introducing him to the algebra studied in
this paper.

\section{The height algebra and depth profiles}\label{sec:height-algebra}

Our tree semantics depends on a tightly constrained interaction between the two
operations in \eqref{eq:height-operations}.  The product distributes over the
join, but is neither associative nor unital.

\begin{proposition}[Basic identities]\label{prop:basic-identities}
For all \(a,b,c,d\in\Ninf\), the following statements hold.
\begin{enumerate}[label=\textnormal{(\roman*)},itemsep=0.15em,topsep=0.3em]
  \item\label{item:basic-join} The operation \(\vee\) is a join-semilattice operation with least
  element \(-\infty\).
  \item\label{item:basic-commutativity} The operation \(\star\) is commutative.
  \item The operation \(\star\) distributes over \(\vee\) in each variable:
  \[
    a\star(b\vee c)=(a\star b)\vee(a\star c).
  \]
  \item\label{item:basic-medial} The operation \(\star\) satisfies the medial identity
  \[
    (a\star b)\star(c\star d)
    =(a\star c)\star(b\star d).
  \]
  \item The operation \(\star\) is not associative.
  \item The operation \(\star\) has no identity element.
  \item The element \(-\infty\) is not absorbing for \(\star\).
\end{enumerate}
\end{proposition}

\begin{proof}
The semilattice laws in \ref{item:basic-join} and commutativity in \ref{item:basic-commutativity} are immediate.
Distributivity is the identity
\begin{align*}
  a\star(b\vee c)
    &=\max\{a,\max\{b,c\}\}+1\\
    &=\max\{\max\{a,b\}+1,\max\{a,c\}+1\}\\
    &=(a\star b)\vee(a\star c),
\end{align*}
and commutativity gives the other variable.  Both sides of \ref{item:basic-medial} reduce to
\(\max\{a,b,c,d\}+2\).  Nonassociativity is witnessed by
\[
  (0\star0)\star1=2,
  \qquad\text{and}\qquad
  0\star(0\star1)=3.
\]
The element \(-\infty\) is not an identity or absorbing because
\((-\infty)\star n=n+1\) for finite \(n\).  No finite \(e\) is an identity,
since that would force \(e=e\star e=e+1\).
\end{proof}

The join in \Cref{def:height-algebra} is not an auxiliary choice.  Once
\(-\infty\) is required to be the least element, the height product recovers
the usual order on \(\Ninf\) and therefore forces the join operation itself.

\begin{theorem}[Uniqueness of the join]\label{thm:unique-join}
We let \(\sqcup\) be a commutative, associative, idempotent operation on
\(\Ninf\) with identity \(-\infty\), and suppose that \(\star\) distributes over
\(\sqcup\).  Then
\[
  a\sqcup b=\max\{a,b\}
\]
for all \(a,b\in\Ninf\).
\end{theorem}

\begin{proof}
We write \(a\preceq b\) when \(a\sqcup b=b\).  This is the semilattice order
of \(\sqcup\), with least element \(-\infty\).  We first recover the usual
order among the positive integers.  If \(0\leq m\leq n\), then
\(n=n\sqcup(-\infty)\), and distributivity gives
\begin{align*}
  n+1
    &=m\star n\\
    &=m\star\bigl(n\sqcup(-\infty)\bigr)\\
    &=(m\star n)\sqcup\bigl(m\star(-\infty)\bigr)\\
    &=(n+1)\sqcup(m+1).
\end{align*}
Thus \(m+1\preceq n+1\).

It remains to place \(0\) below every positive integer.  We fix \(n\geq1\)
and put \(k\coloneqq0\sqcup n\).  Since \(0\preceq k\), the element \(k\)
is finite.  The map
\[
  \tau : x \mapsto (-\infty)\star x
\]
is a \(\sqcup\)-endomorphism by distributivity.  It fixes \(-\infty\) and
sends each finite \(x\) to \(x+1\).  Hence
\[
  k+1
  =\tau(0\sqcup n)
  =1\sqcup(n+1)
  =n+1,
\]
because \(1\preceq n+1\) by the first paragraph.  Therefore
\(0\sqcup n=n\).

We have shown that \(m\sqcup n=n\) whenever
\(-\infty\leq m\leq n\) in the usual order.  Commutativity now gives
\(a\sqcup b=\max\{a,b\}\) for all \(a,b\in\Ninf\).
\end{proof}

\subsection{Rooted trees and depth profiles}

We now turn our attention to the height product.  What does repeated application of
\(\star\) remember about a labelled bracketing?  We continue to fix \(n\geq1\)
and \(X\coloneqq\{x_1,\dots,x_n\}\).  A plane rooted binary tree
distinguishes a left child from a right child at every internal vertex.  We
continue to use the terminology of \Cref{def:tree-monomial}.

The following monomial will serve as a running example for evaluation and
depth profiles.

\begin{example}\label{ex:running-monomial}
We assume \(n\geq3\).  The monomial
\[
  T=x_1\star\bigl((x_1\star x_2)\star x_3\bigr)
\]
has a leaf labelled \(x_1\) as one child of the root and a deeper occurrence
of \(x_1\) in the other root subtree.  We may draw this tree as
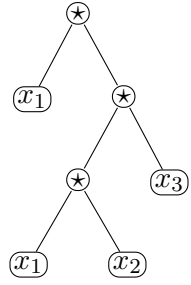
\begin{center}
\begin{forest}
for tree={
  draw,
  rounded corners,
  inner sep=1.5pt,
  s sep=8mm,
  l sep=8mm,
  edge={-}
}
[$\star$
  [$x_1$]
  [$\star$
    [$\star$
      [$x_1$]
      [$x_2$]
    ]
    [$x_3$]
  ]
]
\end{forest}
\end{center}
\end{example}

To turn a labelled tree into an operation, we interpret its labels as input
weights and its internal vertices as applications of the height product.

\begin{definition}[Evaluation]\label{def:tree-evaluation}
For \(u=(u_1,\dots,u_n)\in\Ninf^n\), the \defi{evaluation} of a labelled
tree is defined recursively by
\[
\begin{aligned}
  \eval{x_i}(u)&\coloneqq u_i,\\
  \eval{-\infty}(u)&\coloneqq-\infty,\\
  \text{and}\qquad
  \eval{S\circ R}(u)&\coloneqq\eval{S}(u)\star\eval{R}(u).
\end{aligned}
\]
The resulting value is written \(\eval{T}(u)\).
\end{definition}

Evaluation returns a single number.  The next definition isolates the tree
data that will determine that number for every assignment of weights.

\begin{definition}[Depth profile]\label{def:depth-profile}
For a labelled tree \(T\), we define
\[
  d_i(T)\coloneqq
  \max\{\operatorname{depth}_T(\ell):\ell\text{ is labelled }x_i\}
\]
when \(x_i\) occurs, and we put \(d_i(T)\coloneqq-\infty\) otherwise.  The
vector
\[
  d(T)\coloneqq(d_1(T),\dots,d_n(T))\in\Ninf^n
\]
is the \defi{depth profile} of \(T\).  We write
\[
  \botvec\coloneqq(-\infty,\dots,-\infty)
\]
for the depth profile of every all-silent tree.
\end{definition}

For the running monomial, the profile records only the deepest occurrence of
each variable.

\begin{example}\label{ex:running-monomial-revisited}
For \Cref{ex:running-monomial}, the two occurrences of \(x_1\) have depths \(1\)
and \(3\), while \(x_2\) has depth \(3\) and \(x_3\) has depth \(2\).  Thus,
for \(u=(u_1,u_2,u_3)\in\Ninf^3\),
\[
  d(T)=(3,3,2),
  \qquad\text{and}\qquad
  \eval{T}(u)=\max\{u_1+3,u_2+3,u_3+2\}.
\]
The contribution \(u_1+1\) from the shallower occurrence is dominated by
\(u_1+3\) from the deeper one.
\end{example}

At each grafting, every leaf depth increases by one and the product takes a
maximum.  Iterating this observation gives the identity that controls the
paper.

\begin{theorem}[Weighted leaf-depth formula]\label{thm:weighted-depth}
For every labelled rooted binary tree \(T\) and every \(u\in\Ninf^n\),
\begin{equation}\label{eq:weighted-depth}
  \eval{T}(u)
  =\max_{1\leq i\leq n}\bigl(u_i+d_i(T)\bigr).
\end{equation}
\end{theorem}

\begin{proof}
We induct on \(T\).  The formula holds for a variable leaf and for a silent
leaf.  If \(T=S\circ R\), grafting increases every branch depth by one, so
\[
  d_i(T)=\max\{d_i(S),d_i(R)\}+1.
\]
The induction hypothesis now gives
\begin{align*}
  \eval{T}(u)
    &=\max\{\eval{S}(u),\eval{R}(u)\}+1\\
    &=\max\!\left\{
        \max_i\bigl(u_i+d_i(S)\bigr),
        \max_i\bigl(u_i+d_i(R)\bigr)
      \right\}+1\\
    &=\max_i\bigl(u_i+\max\{d_i(S),d_i(R)\}+1\bigr)\\
    &=\max_i\bigl(u_i+d_i(T)\bigr).
\end{align*}
\end{proof}

The formula immediately turns equality of induced operations into equality
of profiles.

\begin{corollary}[Depth-profile classification]\label{cor:tree-operation-equality}
Two labelled trees induce the same operation \(\Ninf^n\to\Ninf\) if and
only if they have the same depth profile.
\end{corollary}

\begin{proof}
Equal profiles induce equal operations by \eqref{eq:weighted-depth}.
Conversely, we fix \(i\) and evaluate both trees at the assignment that sends
\(x_i\) to \(0\) and every other variable to \(-\infty\).  The resulting
values are their \(i\)th profile coordinates.  Equality of the operations
therefore forces equality of the profiles.
\end{proof}

\subsection{The three-input associativity defect}

The simplest failure of associativity occurs on three inputs.  In
\((a\star b)\star c\), the inputs \(a,b\) lie at depth \(2\) and \(c\) lies
at depth \(1\).  In \(a\star(b\star c)\), the depths of \(a\) and \(c\) are
reversed.  These depth shifts give an exact criterion for when the two
bracketings nevertheless have the same value.

\begin{proposition}[Three-input comparison]\label{prop:associator-formula}
For finite \(a,b,c\),
\begin{align*}
  (a\star b)\star c
    &=\max\{a+2,b+2,c+1\},\\
  a\star(b\star c)
    &=\max\{a+1,b+2,c+2\}.
\end{align*}
The two bracketings agree if and only if either
\(a,c\leq b\) or else \(a=c\).
\end{proposition}

\begin{proof}
Expanding the products gives the displayed formulas.  If \(a=c\), they
agree by symmetry; if \(a,c\leq b\), both equal \(b+2\).

Conversely, we assume that \(a\neq c\) and \(\max\{a,c\}>b\).  If \(a>c\),
then \(a>b\), and
\[
  (a\star b)\star c=a+2,
  \qquad\text{and}\qquad
  a\star(b\star c)=a+1.
\]
The case \(c>a\) is symmetric, so equality occurs only in the stated cases.
\end{proof}

On three inputs, this criterion can be read directly from the two possible
rooted bracketings.

\begin{example}[Three weighted leaves]\label{ex:three-leaves}
The two bracketings on three labelled leaves may be drawn as
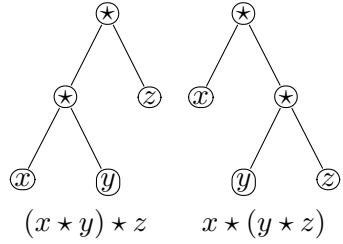
\begin{center}
\begin{tabular}{cc}
\begin{forest}
for tree={draw,rounded corners,inner sep=1.5pt,s sep=8mm,l sep=8mm,edge={-}}
[$\star$
  [$\star$
    [$x$]
    [$y$]
  ]
  [$z$]
]
\end{forest}
&
\begin{forest}
for tree={draw,rounded corners,inner sep=1.5pt,s sep=8mm,l sep=8mm,edge={-}}
[$\star$
  [$x$]
  [$\star$
    [$y$]
    [$z$]
  ]
]
\end{forest}
\\
\((x\star y)\star z\) & \(x\star(y\star z)\)
\end{tabular}
\end{center}
Their profiles are
\[
  d((x\star y)\star z)=(2,2,1),
  \qquad\text{and}\qquad
  d(x\star(y\star z))=(1,2,2).
\]
By \eqref{eq:weighted-depth}, they induce the max-plus linear forms
\[
  \max\{u_x+2,u_y+2,u_z+1\}
  \quad\text{and}\quad
  \max\{u_x+1,u_y+2,u_z+2\},
\]
respectively.  Thus the comparison records the different labelled depths of
the two rooted trees.
\end{example}

The preceding example shows how different depth profiles distinguish
bracketings.  Conversely, \eqref{eq:weighted-depth} discards sibling
relations, planar order, and all but the greatest depth of each label.  The
next example exhibits two genuinely different trees that the profile quotient
therefore identifies.  When every leaf is active and assigned \(0\), this
specialization recovers ordinary height.

\begin{example}[Different trees with one shadow]\label{ex:same-shadow}
The labelled trees
\[
  (x\star y)\star(x\star z)
  \qquad\text{and}\qquad
  (x\star x)\star(y\star z)
\]
may be drawn as
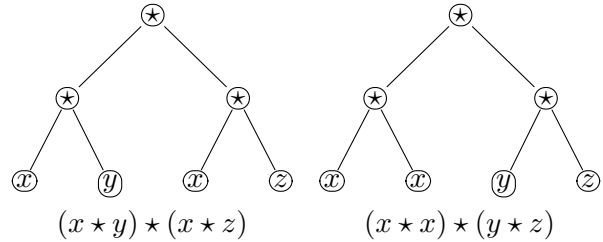
\begin{center}
\begin{tabular}{cc}
\begin{forest}
for tree={draw,rounded corners,inner sep=1.5pt,s sep=8mm,l sep=8mm,edge={-}}
[$\star$
  [$\star$
    [$x$]
    [$y$]
  ]
  [$\star$
    [$x$]
    [$z$]
  ]
]
\end{forest}
&
\begin{forest}
for tree={draw,rounded corners,inner sep=1.5pt,s sep=8mm,l sep=8mm,edge={-}}
[$\star$
  [$\star$
    [$x$]
    [$x$]
  ]
  [$\star$
    [$y$]
    [$z$]
  ]
]
\end{forest}
\\
\((x\star y)\star(x\star z)\) & \((x\star x)\star(y\star z)\)
\end{tabular}
\end{center}
They are different even after planar order is forgotten, since their root subtrees
carry different multisets of labels.  Nevertheless, both have depth profile
\((2,2,2)\) and therefore induce the same operation
\[
  (u_x,u_y,u_z)\longmapsto\max\{u_x+2,u_y+2,u_z+2\}.
\]
Thus the induced operation remembers the deepest occurrence of each label
and nothing about how those occurrences are paired.
\end{example}

\section{The Kraft region of tree monomials}\label{sec:kraft}

The weighted depth formula turns a tree monomial into a vector specifying the
greatest depth assigned to each label.  Not every such vector can come from
one tree: one leaf at each prescribed depth must fit into a common binary
branching structure.  The remaining constraint is therefore geometric.

\begin{definition}[Kraft weight and region]\label{def:kraft-region}
We adopt the convention \(2^{-(-\infty)}=0\).  We define the \defi{Kraft weight}
and \defi{Kraft region} of \(d=(d_1,\dots,d_n)\in\Ninf^n\) by
\begin{equation}\label{eq:kraft-weight}
  \kappa(d)\coloneqq
  \sum_{i=1}^n2^{-d_i},
  \qquad\text{and}\qquad
  \Kraft_n\coloneqq
  \{d\in\Ninf^n:\kappa(d)\leq1\}.
\end{equation}
\end{definition}

The first examples show both how silent leaves fill unused capacity and how
the inequality can fail.

\begin{example}[First Kraft profiles]\label{ex:first-kraft-profiles}
In three variables, \((1,1,-\infty)\) has Kraft weight \(1\) and is realized
by \(x_1\star x_2\).  The profile \((2,2,2)\) has weight \(3/4\) and is
realized by \((x_1\star x_2)\star(x_3\star(-\infty))\).  The silent leaf fills
the unused quarter of the binary tree.  By contrast, \((1,1,1)\) has weight
\(3/2\), so three labels cannot all occur at depth \(1\) in one binary tree.
\end{example}

The proof of the general criterion encodes root-to-leaf paths as binary
words.  Disjointness of those paths gives necessity; in the converse
direction, unused branches are completed with silent leaves.

\begin{theorem}[Monomial profile criterion]\label{thm:kraft-profile}
A vector \(d\in\Ninf^n\) is the depth profile of a labelled rooted binary
tree if and only if \(d\in\Kraft_n\).  Moreover, the tree may be chosen
so that each variable that occurs appears exactly once; every remaining leaf may be
given the silent label \(-\infty\).
\end{theorem}

\begin{proof}
For necessity, we suppose that a labelled tree realizes \(d\).  We choose a planar embedding
and record each root-to-leaf path as the binary word obtained by writing
\(0\) for a left edge and \(1\) for a right edge.  We choose one occurrence of
each occurring variable at depth \(d_i\).  The resulting words are prefix-free, i.e.
no chosen word is a prefix of another.
We associate to a word of length \(\ell\) the dyadic subinterval of
\([0,1)\) consisting of binary expansions that begin with that word.  The
intervals associated with the chosen words are disjoint and have lengths
\(2^{-d_i}\).  Consequently, the disjointness of these intervals gives
\(\kappa(d)\leq1\).

For sufficiency, we suppose that \(d\in\Kraft_n\).  If \(d=\botvec\), one silent leaf realizes \(d\).  Otherwise, we write the finite
coordinates in nondecreasing order as
\[
  \ell_1\leq\ell_2\leq\cdots\leq\ell_m.
\]
We construct prefix-free binary words of these lengths.  We assume that words of
lengths \(\ell_1,\dots,\ell_{j-1}\) have been chosen.  A word of length
\(\ell_i\) excludes exactly \(2^{\ell_j-\ell_i}\) words of length
\(\ell_j\).  Since
\[
  \sum_{i<j}2^{-\ell_i}\leq1-2^{-\ell_j},
\]
the previously chosen words exclude at most \(2^{\ell_j}-1\) words of length
\(\ell_j\).  At least one remains, so the construction continues.

The prefix tree of the code has as vertices all prefixes of the
chosen words, with an edge corresponding to the addition of one binary digit.
We take this tree.  If an internal vertex has only one child, we add a silent leaf
as its other child.  We label the selected code leaves by the corresponding
variables and then forget the planar embedding.
Each occurring variable now appears exactly once at its prescribed depth, while
the added leaves do not affect the profile.
\end{proof}

The equality case isolates the familiar situation in which no silent leaves
are needed.

\begin{remark}\label{rem:kraft-equality}
If silent leaves are disallowed and every variable occurs exactly once, then
\(\kappa(d)=1\).  Allowing silent leaves replaces this equality by the
inequality in \eqref{eq:kraft-weight}; this permits partial depth data.
\end{remark}

The join and product on \(\Ninf\) induce coordinatewise operations on
\(\Ninf^n\):
\begin{equation}\label{eq:profile-operations}
  (a\vee b)_i\coloneqq\max\{a_i,b_i\},
  \qquad\text{and}\qquad
  (a\star b)_i\coloneqq\max\{a_i,b_i\}+1.
\end{equation}
The second operation is the profile of grafting two labelled trees beneath a
new root.

\begin{proposition}[Closure of the Kraft region]\label{prop:kraft-closure}
The Kraft region \(\Kraft_n\) is closed under the coordinatewise product
in \eqref{eq:profile-operations}.
\end{proposition}

\begin{proof}
For each coordinate, we have the inequality
\[
  2^{-(a\star b)_i}
  \leq
  \frac12\bigl(2^{-a_i}+2^{-b_i}\bigr).
\]
Therefore
\[
  \kappa(a\star b)
  \leq
  \frac12\bigl(\kappa(a)+\kappa(b)\bigr)
  \leq1,
\]
whenever \(a,b\in\Kraft_n\).
\end{proof}

For \(1\leq i\leq n\), we put
\[
  e_i\coloneqq
  (-\infty,\dots,-\infty,0,-\infty,\dots,-\infty),
\]
with \(0\) in coordinate \(i\).  This is the depth profile of the
one-leaf tree labelled \(x_i\).  These one-leaf profiles are the natural
candidates for profiles that cannot be obtained by grafting.  The Kraft
criterion detects exactly this obstruction: a profile fails to factor when
one of its prescribed leaves is forced to sit at the root.

\begin{proposition}[Root factorization]\label{prop:root-factorization}
For \(d\in\Kraft_n\), the following are equivalent:
\begin{enumerate}[label=(\roman*)]
  \item\label{item:root-factorization-product} \(d=a\star b\) for some \(a,b\in\Kraft_n\);
  \item\label{item:root-factorization-positive} no coordinate of \(d\) is \(0\);
  \item\label{item:root-factorization-tree} \(d\) is represented by a tree with an internal root.
\end{enumerate}
Consequently, the only elements of \(\Kraft_n\) that cannot be written
as \(a\star b\) with \(a,b\in\Kraft_n\) are \(e_1,\dots,e_n\).
\end{proposition}

\begin{proof}
We prove \(\ref{item:root-factorization-product}\Rightarrow\ref{item:root-factorization-positive}\Rightarrow\ref{item:root-factorization-tree}\Rightarrow\ref{item:root-factorization-product}\).  If
\(d=a\star b\), every finite coordinate of \(d\) is positive, proving \ref{item:root-factorization-positive}.

We assume \ref{item:root-factorization-positive}.  If \(d=\botvec\), two silent leaves beneath a root represent
it.  Otherwise every finite coordinate is positive, so the prefix-code
construction of \Cref{thm:kraft-profile} uses no empty codeword.  Its tree
therefore has an internal root, proving \ref{item:root-factorization-tree}.

If \ref{item:root-factorization-tree} holds, the two root subtrees have profiles
\(a,b\in\Kraft_n\), and \(d=a\star b\), proving \ref{item:root-factorization-product}.

Finally, an admissible profile with a zero coordinate has all other
coordinates equal to \(-\infty\) by the Kraft inequality, so it is one of
\(e_1,\dots,e_n\).  The equivalence shows that these are the only profiles that do not factor.
\end{proof}

Thus \(\Kraft_n\) is the coefficient region of single tree
monomials.  Its Kraft inequality records the prefix-code obstruction imposed
by one rooted binary tree.

The same criterion also controls a basic finiteness distinction inside the
fibres of the depth-profile map.  Such fibres can contain inessential variation
obtained by expanding silent subtrees.  To separate that variation from the
remaining geometry, we first remove these silent expansions.

\begin{definition}[Silent reduction]\label{def:silent-reduction}
A labelled tree is \defi{silent-reduced} if every all-silent rooted subtree is
a single leaf.  The \defi{silent reduction} of a labelled tree \(T\) is
obtained by collapsing each maximal all-silent subtree to a single silent
leaf.
\end{definition}

\begin{definition}[Raw and reduced fibres]\label{def:raw-reduced-fibres}
We let \(\Trees(X)\) denote the set of finite full labelled nonplane rooted
binary trees over \(X\).  The assignment
\[
  T\longmapsto d(T)
\]
defines the \defi{depth-profile map} from \(\Trees(X)\) to \(\Kraft_n\).  For
\(a\in\Kraft_n\), the \defi{raw fibre over \(a\)} is
\[
  \mathcal F^{\mathrm{raw}}_a
  \coloneqq
  \{T\in\Trees(X):d(T)=a\},
\]
and the \defi{reduced fibre over \(a\)} is
\[
  \mathcal F^{\mathrm{red}}_a
  \coloneqq
  \{T\in\mathcal F^{\mathrm{raw}}_a:T\text{ is silent-reduced}\}.
\]
The \defi{kernel equivalence} of the depth-profile map is
\[
  S\sim_{\mathrm{prof}}T
  \quad\Longleftrightarrow\quad
  d(S)=d(T);
\]
its equivalence classes are precisely the raw fibres.
\end{definition}

Silent reduction preserves the depth profile.  The two fibres separate
variation caused by silent expansions from the remaining geometry.

\begin{proposition}[Finiteness of raw and reduced fibres]\label{prop:raw-reduced-fibres}
We fix \(n\geq1\) and \(d\in\Kraft_n\).
\begin{enumerate}[label=\textnormal{(\arabic*)},leftmargin=*]
  \item\label{item:reduced-fibre-finite} The reduced fibre over \(d\) is finite.
  \item The raw fibre over \(d\) is finite if \(\kappa(d)=1\) and countably
  infinite if \(\kappa(d)<1\).
\end{enumerate}
\end{proposition}

\begin{proof}
We first suppose that \(d=\botvec\).  The only silent-reduced representative is
one silent leaf.  Otherwise, we let \(D\) be the largest finite coordinate of
\(d\).  In a silent-reduced tree, the sibling subtree of every silent leaf
contains a variable-labelled leaf.  Hence no silent leaf can lie below depth
\(D\), and the whole tree has height at most \(D\).  There are only finitely
many nonplane rooted binary trees of bounded height and only finitely many
labellings of their leaves by \(x_1,\dots,x_n,-\infty\).  This proves \ref{item:reduced-fibre-finite}.

If \(\kappa(d)<1\), the prefix-code construction in
\Cref{thm:kraft-profile} produces a representative with a silent leaf.
Replacing that leaf by any finite all-silent rooted binary tree gives a distinct
representative with the same profile.  There are countably many finite rooted
binary trees, so the raw fibre is countably infinite.

We now suppose that \(\kappa(d)=1\).  In any representative, we choose one leaf at
depth \(d_i\) for each variable \(x_i\) that occurs.  Their Kraft weights
already sum to one.  Since the full leaf set of a finite binary tree also has
Kraft sum one, these chosen leaves are all the leaves of the tree.  In
particular, no silent or additional variable-labelled leaves occur.  The tree
has height at most \(\max_i d_i\), so only finitely many representatives are
possible.
\end{proof}

Passing to reduced fibres removes the automatic infinitude caused by silent
expansions and leaves a finite enumeration problem for every fixed profile.
Its exact enumeration remains open.

\section{Tree polynomials and the tropical coefficient algebra}\label{sec:normal-forms}

We now pass from single tree monomials to arbitrary terms, showing that finite
joins remove the Kraft obstruction and produce the full coefficient algebra.

\subsection{The coefficient algebra}

A single monomial contributes one vector in the Kraft region.  Finite joins
combine such vectors coordinatewise and remove the prefix-code obstruction.
The induced operations remain max-plus linear, but now every coefficient vector
becomes available.

In an ordinary commutative associative algebra, a monomial records the
multiplicity of each variable.
A rooted-tree monomial instead records the greatest depth of each variable.
Repeated occurrences matter only when one lies deeper than the previous
ones, and a change of bracketing can change the coefficient vector without
changing the multiset of leaves.  The profile forgets sibling relations,
planar order, and shallower repetitions, but it retains the deepest labelled
occurrences.

The Kraft region \(\Kraft_n\) is the coefficient space of individual
monomials.  It is closed under grafting by \Cref{prop:kraft-closure}, but not
under coordinatewise join.  However, passing to tree polynomials fills all of
\(\Ninf^n\).  Join therefore enlarges the coefficient region without
producing operations beyond max-plus linear forms.

\begin{definition}[Tree polynomial and profile]
\label{def:tree-polynomial-profile}
A \defi{tree polynomial} is a term built from variables, the constant
\(-\infty\), the join \(\vee\), and the product \(\star\).  Its
\defi{profile} \(d(t)\in\Ninf^n\) is defined recursively by
\begin{equation}\label{eq:polynomial-profile-recursion}
\begin{aligned}
  d(-\infty)&\coloneqq\botvec,\\
  d(x_i)_j&\coloneqq
  \begin{cases}
    0,&i=j,\\
    -\infty,&i\neq j,
  \end{cases}\\
  d(s\vee t)&\coloneqq d(s)\vee d(t),\\
  \text{and}\qquad
  d(s\star t)&\coloneqq d(s)\star d(t).
\end{aligned}
\end{equation}
Here the last two operations are those of \eqref{eq:profile-operations}.
\end{definition}

Every coefficient profile determines a canonical operation on input weights.

\begin{definition}[Max-plus linear form]\label{def:max-plus-linear-form}
For \(d\in\Ninf^n\), we define the \defi{max-plus linear form}
\begin{equation}\label{eq:tropical-linear-form}
  \lambda_d(u_1,\dots,u_n)
  \coloneqq
  \max_{1\leq i\leq n}(u_i+d_i).
\end{equation}
\end{definition}

The adjective ``tropical'' refers to the max-plus linear forms in
\eqref{eq:tropical-linear-form}.  Equation \eqref{eq:profile-operations}
identifies join and grafting with coordinatewise operations on their
coefficient vectors.

The recursive profile and its linear form agree under evaluation.

\begin{theorem}[Tropical normal form]\label{thm:tropical-normal-form}
Every \(n\)-variable tree polynomial \(t\) induces the operation
\[
  \lambda_{d(t)}(u_1,\dots,u_n).
\]
Two tree polynomials induce the same operation on \(\Hh\) if and only if
their profiles agree.
\end{theorem}

\begin{proof}
The normal form follows by structural induction.  Variables and the constant
\(-\infty\) have the profiles prescribed in
\eqref{eq:polynomial-profile-recursion}.  If \(s\) and \(t\) induce
\(\lambda_a\) and \(\lambda_b\), then
\[
\begin{aligned}
  (\lambda_a\vee\lambda_b)(u)
  &=\max_i\bigl(u_i+\max\{a_i,b_i\}\bigr)
    =\lambda_{a\vee b}(u),\\
  (\lambda_a\star\lambda_b)(u)
  &=\max\{\lambda_a(u),\lambda_b(u)\}+1
    =\lambda_{a\star b}(u).
\end{aligned}
\]
Thus the recursive profile tracks both term-forming operations.

To prove uniqueness, we isolate one coordinate at a time.  On the assignment
\(u_i=0\) and \(u_j=-\infty\) for \(j\neq i\), the value of \(\lambda_d\)
is \(d_i\).  The induced operation therefore determines every coordinate of
\(d\).
\end{proof}

\begin{definition}[Term-operation algebra]\label{def:term-operation-algebra}
We let \(\operatorname{TermOp}_n(\Hh)\) be the set of all operations
\(\Hh^n\to\Hh\) induced by \(n\)-variable terms in the signature
\((\vee,\star,-\infty)\).  We equip this set with pointwise join and product and
with the constant operation of value \(-\infty\).  Equivalently,
\(\operatorname{TermOp}_n(\Hh)\) is the quotient of the \(n\)-variable term
algebra by the equivalence relation that identifies two terms exactly when
they induce the same operation on \(\Hh\).
\end{definition}

The normal form is complete once every coefficient vector is shown to occur.

\begin{theorem}[Profile realization and coefficient algebra]\label{thm:all-profiles}
Every vector in \(\Ninf^n\) is the profile of a tree polynomial.  Under
profile evaluation, \(\operatorname{TermOp}_n(\Hh)\) is naturally isomorphic
to the direct power \(\Hh^n\).
\end{theorem}

\begin{proof}
We first realize an arbitrary coefficient vector.  For a term \(q\), we define
\[
  s^0(q)\coloneqq q,
  \qquad\text{and}\qquad
  s^{k+1}(q)\coloneqq s^k(q)\star s^k(q).
\]
The profile of \(s^k(x_i)\) has value \(k\) in coordinate \(i\) and
\(-\infty\) elsewhere.  For \(a\in\Ninf^n\), we join the terms
\(s^{a_i}(x_i)\) over the finite coordinates; if none is finite, we use
\(-\infty\).  The resulting profile is \(a\).

This makes the profile map surjective, while
\Cref{thm:tropical-normal-form} makes it injective on term operations.
The recursion \eqref{eq:polynomial-profile-recursion} shows that it preserves
both operations.  Hence the term-operation algebra is naturally isomorphic
to \(\Hh^n\).
\end{proof}

\begin{remark}[Formal polynomials and polynomial functions]
Conventional tropical algebra distinguishes formal polynomials from the
functions they induce; distinct formal polynomials may represent the same
function.  Here the corresponding distinction is between the tree
polynomials of \Cref{def:tree-polynomial-profile} and the max-plus linear
forms of \Cref{def:max-plus-linear-form}.  By
\Cref{thm:tropical-normal-form}, equality of the induced maps is exactly
equality of profiles, and
\Cref{def:term-operation-algebra,thm:all-profiles} identifies their semantic
quotient with \(\Hh^n\).
\end{remark}

The first profile realized by a tree polynomial but not by a single tree monomial appears in three variables.

\begin{example}[A polynomial outside the Kraft region]\label{ex:outside-kraft}
The polynomial
\[
  (x\star x)\vee(y\star y)\vee(z\star z)
\]
has profile \((1,1,1)\), whose Kraft weight is \(3/2\).  No single tree
monomial has this profile, although the polynomial induces
\[
  \lambda_{(1,1,1)}(u_x,u_y,u_z)
  =\max\{u_x+1,u_y+1,u_z+1\}.
\]
Thus join removes the monomial obstruction while preserving the tropical
normal form.
\end{example}

\subsection{Relatively free algebras and the word problem}

The normal form also identifies the relatively free algebras in the variety
\(\Variety(\Hh)\) generated by \(\Hh\).

\begin{corollary}[Relatively free algebras]\label{cor:free-algebras}
The free algebra on \(n\) generators in \(\Variety(\Hh)\) is naturally
isomorphic to \(\Hh^n\), with the generators represented by the coordinate
vectors \(e_1,\dots,e_n\).
\end{corollary}

\begin{proof}
Under \Cref{thm:all-profiles}, the variable \(x_i\) has profile \(e_i\).
We let \(A\in\Variety(\Hh)\), choose \(a_1,\dots,a_n\in A\), and for
\(d\in\Hh^n\) select a term \(t\) with \(d(t)=d\).  We define
\[
  \varphi(d)\coloneqq t^A(a_1,\dots,a_n).
\]
If \(d(s)=d(t)\), then \Cref{thm:tropical-normal-form} gives the identity
\(s\approx t\) in \(\Hh\), hence in \(A\); thus \(\varphi\) is
well-defined.  Term evaluation makes it a homomorphism with
\(\varphi(e_i)=a_i\).  Since the \(e_i\) generate \(\Hh^n\), these values
also make \(\varphi\) unique.
\end{proof}

The free-algebra description also yields an efficient equality test for
terms.

\begin{corollary}[Word problem]\label{cor:word-problem}
The equational theory of \(\Hh\) is decidable by a bottom-up profile
calculation.  For a fixed finite variable set, equality of two terms can be
tested in time linear in their syntax-tree size, assuming constant-time
arithmetic on integers of \(O(\log |t|)\) bits.
\end{corollary}

\begin{proof}
We compute \(d(t)\) at each syntax node using
\eqref{eq:polynomial-profile-recursion}.  For fixed arity, every profile has
constant length, and each finite entry is at most the number of product nodes
below it.  Thus each update takes constant time and every entry fits in one
machine word.  By \Cref{thm:tropical-normal-form}, two terms agree on \(\Hh\)
if and only if their root profiles agree.
\end{proof}

\section{Finite tree languages and completion}\label{sec:tree-languages}

By distributivity, every tree polynomial induces the same operation as a
finite join of tree monomials; associativity, commutativity, and idempotence of
\(\vee\) allow those monomials to be collected as a finite set.  Its profile is
the coordinatewise maximum of their profiles.  Finite tree languages therefore
give an exact model for the passage from the Kraft region to the full
coefficient algebra.

\subsection{Nonempty languages}

\begin{definition}[Finite nonempty tree language]
\label{def:nonempty-tree-language}
We recall that \(\Trees(X)\) is the set of labelled rooted binary trees from
\Cref{def:raw-reduced-fibres}.  We define
\begin{equation}\label{eq:nonempty-language-operations}
\begin{aligned}
  \Languages^+(X)
    &\coloneqq
    \{L\subseteq\Trees(X):0<\lvert L\rvert<\infty\},\\
  L\vee M&\coloneqq L\cup M,\\
  \text{and}\qquad
  L\circ M&\coloneqq\{S\circ T:S\in L,\ T\in M\}.
\end{aligned}
\end{equation}
An element of \(\Languages^+(X)\) is a \defi{finite nonempty tree language}.
Thus \(\Languages^+(X)\) is equipped with union and pairwise grafting.
\end{definition}

Evaluation of a language combines the contributions of all of its trees by
taking their join.

\begin{definition}[Language profile and weighted-height function]
\label{def:language-profile}
For \(L\in\Languages^+(X)\), we define its \defi{language profile} and
\defi{weighted-height function} by
\begin{equation}\label{eq:language-profile}
  D(L)\coloneqq\bigvee_{T\in L}d(T),
  \qquad\text{and}\qquad
  \eval{L}(u)\coloneqq\bigvee_{T\in L}\eval{T}(u)
  =\lambda_{D(L)}(u).
\end{equation}
The last equality follows from \eqref{eq:weighted-depth} by taking the
coordinatewise maximum over \(T\in L\).
\end{definition}

Different finite languages can therefore have the same profile and induce
the same operation.

\begin{example}[Different languages with one profile]
\label{ex:language-profile}
We take \(X=\{x,y,z\}\), and let
\[
  L=\{x\star y\},
  \qquad\text{and}\qquad
  M=\{x\star x,\ y\star y\}.
\]
Although \(L\neq M\), both have profile \((1,1,-\infty)\) and induce
\[
  (u_x,u_y,u_z)\longmapsto\max\{u_x+1,u_y+1\}.
\]
Grafting either language with \(\{z\}\) gives profile \((2,2,1)\), so their
agreement survives the language product even though the resulting sets of
trees remain different.
\end{example}

Union therefore already descends to coordinatewise join.  The only further
point is that pairwise grafting descends to the coordinatewise product.

\begin{theorem}[Nonempty tree languages]\label{thm:nonempty-languages}
The profile map
\[
  D\colon\Languages^+(X)\longrightarrow\Hh^n
\]
is a surjective homomorphism from union and pairwise grafting to
coordinatewise join and product.  For \(L,M\in\Languages^+(X)\), the
following conditions are equivalent:
\[
  D(L)=D(M),
  \qquad\text{and}\qquad
  \eval{L}=\eval{M}.
\]
\end{theorem}

\begin{proof}
The language profile tracks union directly:
\[
  D(L\vee M)=D(L)\vee D(M).
\]
For grafting, every leaf acquires one additional unit of depth, so
\begin{align*}
  D(L\circ M)_i
    &=\max_{S\in L,\,T\in M}
      \bigl(\max\{d_i(S),d_i(T)\}+1\bigr)\\
    &=\max\{D(L)_i,D(M)_i\}+1\\
    &=(D(L)\star D(M))_i.
\end{align*}
Thus \(D\) is a homomorphism.  The construction in
\Cref{thm:all-profiles} proves surjectivity.  Finally,
\eqref{eq:language-profile} gives \(\eval{L}=\lambda_{D(L)}\), and
\Cref{thm:tropical-normal-form} shows that these functions agree if and only
if their profiles agree.
\end{proof}

\subsection{The empty language}

The empty language is not represented by the all-\(-\infty\) profile.  That
profile already belongs to a nonempty language consisting only of a silent
tree, and it is not absorbing under the coordinatewise product.  The empty
language instead acts as an identity for union and as an absorbing element
for pairwise grafting.

\begin{definition}[Empty-language completion]\label{def:empty-completion}
We let \(A\) be a set equipped with a join-semilattice operation \(\vee\) and a
binary operation \(\star\).  Its \defi{empty-language completion} is
\[
  A^{\varnothing}\coloneqq\{\mathbf 0\}\sqcup A,
\]
where \(\mathbf 0\) is a new least element for \(\vee\) and an absorbing
element for \(\star\):
\[
  \mathbf 0\vee a\coloneqq a,
  \qquad
  \mathbf 0\star a\coloneqq\mathbf 0,
  \qquad\text{and}\qquad
  a\star\mathbf 0\coloneqq\mathbf 0.
\]
\end{definition}

We now allow the language itself to be empty.  We let
\[
  \Languages(X)
  \coloneqq
  \{L\subseteq\Trees(X):\lvert L\rvert<\infty\}
\]
be the set of all finite tree languages.  The operations in
\eqref{eq:nonempty-language-operations} extend to
\(\Languages(X)\), now including the empty language.  The profile map from
nonempty languages should therefore extend by sending the empty language to
the new bottom element.  We apply \Cref{def:empty-completion} to the
\((\vee,\star)\)-reduct of \(\Hh^n\).

\begin{corollary}[Finite-language quotient]\label{cor:empty-language}
The map
\[
  \overline D\colon\Languages(X)\longrightarrow
  \bigl(\Hh^n\bigr)^{\varnothing},
  \qquad
  \overline D(L)\coloneqq
  \begin{cases}
    \mathbf 0,&L=\varnothing,\\
    D(L),&L\neq\varnothing,
  \end{cases}
\]
is a surjective homomorphism.  The empty language is alone in its fibre, and
two nonempty languages have the same image if and only if they induce the same
weighted-height function.
\end{corollary}

\begin{proof}
For nonempty languages, the assertion is \Cref{thm:nonempty-languages}.
Union with \(\varnothing\) fixes a language, whereas pairwise grafting with
\(\varnothing\) yields \(\varnothing\).  These are the defining laws
of \(\mathbf 0\) in \Cref{def:empty-completion}.  Surjectivity and the fibre
statements now follow from the nonempty case.
\end{proof}

Tree-language semantics has therefore forced a new bottom element.  It
records the absence of a tree, not merely the absence of a variable-labelled
leaf.  The notation \(\mathbf 0\) is global: it denotes one new element below
every profile in \(\Hh^n\), not a tuple of coordinatewise scalar bottoms.  We
reserve \(\bot\) below for the absorbing state in a scalar completion of
\(\Hh\).

\subsection{A unital and zero-absorbing completion}

The empty-language bottom is absorbing for grafting, but the height product
still has no unit.  We use \defi{formal empty tree} only as a heuristic name
for a grafting unit: it is not an element of \(\Trees(X)\), and we do not
enlarge \(\Trees(X)\) by an actual empty tree.  This algebraic unit must be
distinct both from the absent-language state and from the one-leaf silent
tree.  Adjoining this second boundary state produces the following
completion.

\begin{proposition}[Unital and zero-absorbing completion]
\label{prop:height-unital-zero-completion}
We define
\[
  \wideNinf\coloneqq\{\bot,\varepsilon\}\sqcup\Ninf
\]
with chain order
\[
  \bot<\varepsilon<-\infty<0<1<2<\cdots.
\]
We let \(\vee\) be maximum in this order, and define
\[
\begin{aligned}
  \bot\odot x=x\odot\bot&\coloneqq\bot
    &&\text{for}\quad x\in\wideNinf,\\
  \varepsilon\odot x=x\odot\varepsilon&\coloneqq x
    &&\text{for}\quad x\in\{\varepsilon\}\sqcup\Ninf,\\
  \text{and}\qquad a\odot b&\coloneqq a\star b
    &&\text{for}\quad a,b\in\Ninf.
\end{aligned}
\]
Then \((\wideNinf,\vee,\bot)\) is a join-semilattice with least element,
\((\wideNinf,\odot,\varepsilon)\) is a commutative unital magma, \(\bot\) is
absorbing for \(\odot\), and \(\odot\) distributes over finite joins.  We write
\[
  \wideHh
  \coloneqq
  (\wideNinf,\vee,\odot,\bot,\varepsilon).
\]
Moreover, \(\Ninf\subseteq\wideNinf\) is closed under \(\vee\) and \(\odot\),
and the induced algebra
\[
  (\Ninf,\vee,\odot|_{\Ninf^2},-\infty)
\]
is \(\Hh\).
\end{proposition}

\begin{proof}
Commutativity, the unit law for \(\varepsilon\), and absorption by \(\bot\)
are immediate.  For distributivity, it suffices on the chain
\(\wideNinf\) to show that every translation is order-preserving.
Translations by \(\bot\) and \(\varepsilon\) are respectively constant and
the identity.  For \(a\in\Ninf\), translation by \(a\) is
\[
  \bot\longmapsto\bot,
  \qquad
  \varepsilon\longmapsto a,
  \qquad\text{and}\qquad
  x\longmapsto\max\{a,x\}+1
  \quad\text{for}\quad x\in\Ninf.
\]
The last rule is order-preserving, and the new comparisons are respected
because \(\bot<a\leq a\star(-\infty)\).  Hence \(\odot\) distributes over
\(\vee\).  The set \(\Ninf\) is closed under both operations, and the
restriction of \(\odot\) to \(\Ninf^2\) is \(\star\), proving the final
assertion.
\end{proof}

The two adjoined elements satisfy different multiplicative laws, and those
laws encode different tree-language semantics.

\begin{example}[Three semantic boundary states]
\label{ex:three-boundary-states}
For every finite \(a\in\NN\), the three lower states obey
\[
\begin{alignedat}{3}
  \bot\vee a&=a, &\qquad \bot\odot a&=\bot
    &\quad&\text{(no tree)},\\
  \varepsilon\vee a&=a, & \varepsilon\odot a&=a
    &&\text{(heuristic empty-tree unit)},\\
  \text{and}\qquad (-\infty)\vee a&=a,
    & (-\infty)\star a&=a+1
    &&\text{(all-silent tree)}.
\end{alignedat}
\]
The last law reflects genuine grafting: placing a silent leaf beside a tree
raises every active leaf depth by one.  Thus the all-silent tree is neither
absent nor a grafting unit.  The finite-language quotient \((\Hh^n)^{\varnothing}\) adjoins the global
state \(\mathbf 0\).  The scalar completion \(\wideHh\) uses \(\bot\) for
an analogous absorbing state and also adjoins \(\varepsilon\).
\end{example}

The native height algebra \(\Hh\), rather than its completion, remains the
subject of this paper.  For \(\Hh\) itself, uniqueness of the join leaves the
successor-chain classifications developed next.

\section{Subalgebras, endomorphisms, and finite quotients}\label{sec:structure}

The product has already forced the compatible join.  The remaining intrinsic
classifications reduce to the successor chain above \(-\infty\): a least
finite element controls a subalgebra, the image of \(0\) controls an
endomorphism, and the first nontrivial congruence class controls a quotient.

\subsection{Subalgebras and endomorphisms}

From any finite \(m\), repeated self-products generate the entire tail
\(m,m+1,m+2,\dots\).  A subalgebra containing a finite element is therefore
determined by its least one.

\begin{theorem}[Subalgebras]\label{thm:subalgebras}
In the signature \((\vee,\star,-\infty)\), the subalgebras of \(\Hh\) are
\[
  \{-\infty\}
  \qquad\text{and}\qquad
  \{-\infty\}\cup\{m,m+1,m+2,\dots\}
  \quad\text{for}\quad m\in\NN.
\]
\end{theorem}

\begin{proof}
Every subalgebra contains \(-\infty\).  If \(m\) is its least finite element,
then \(m\star m=m+1\) generates every integer at least \(m\), while minimality
excludes the smaller finite integers.  Conversely, each displayed tail is
closed under maximum and the successor-of-maximum product.
\end{proof}

The recursion \(n+1=n\star n\) now shows that an endomorphism is determined
by the image of \(0\).

\begin{theorem}[Endomorphisms]\label{thm:endomorphisms}
Every endomorphism of \(\Hh\) is either the constant map to \(-\infty\) or
a translation \(\tau_c\), where
\[
  \tau_c(x)\coloneqq
  \begin{cases}
    -\infty,&x=-\infty,\\
    x+c,&x\in\NN,
  \end{cases}
\]
for a unique \(c\in\NN\).  Hence \(\Hh\) is rigid.
\end{theorem}

\begin{proof}
We let \(f\) be an endomorphism.  It fixes \(-\infty\), and
\(n+1=n\star n\) shows that it is determined by \(f(0)\).  If
\(f(0)=-\infty\), then \(f\) is constant.  If \(f(0)=c\in\NN\), the same
recursion gives \(f(n)=n+c\).

The constant map preserves both operations.  Each \(\tau_c\) preserves
maximum, and for finite \(a,b\),
\[
  \tau_c(a\star b)=\max\{a,b\}+1+c
  =\max\{a+c,b+c\}+1
  =\tau_c(a)\star\tau_c(b).
\]
The cases involving \(-\infty\) follow from the defining convention.  Thus
all listed maps are endomorphisms, and only \(\tau_0\) is surjective.
\end{proof}

\subsection{Congruences and finite quotients}

The same successor mechanism makes every nontrivial identification among
finite heights propagate upward.  We recall that a congruence on \(\Hh\) is an
equivalence relation compatible with \(\vee\), \(\star\), and the named
constant \(-\infty\).  The candidate relations collapse one final tail at a
time.

\begin{definition}[Threshold relations]\label{def:threshold-relations}
For \(m\in\NN\), we put \(T_m\coloneqq\{m,m+1,m+2,\dots\}\), and define
the equivalence relation \(\theta_m\) by
\begin{equation}\label{eq:threshold-congruence}
  x\mathrel{\theta_m}y
  \quad\Longleftrightarrow\quad
  x=y
  \quad\text{or}\quad
  x,y\in T_m.
\end{equation}
Thus \(T_m\) is the only nonsingleton \(\theta_m\)-class, while \(-\infty\)
remains a singleton.  We let \(\Delta\) denote equality and \(\nabla\) the
universal relation on \(\Ninf\).
\end{definition}

Once two finite heights are identified, successor and convexity collapse the
entire tail above their first identification.  The next theorem shows that the
threshold relations exhaust the possibilities.

\begin{theorem}[Congruence classification]\label{thm:congruences}
The congruences of \(\Hh\) are
\[
  \Delta,
  \qquad
  \theta_m\quad\text{for}\quad m\in\NN,
  \qquad\text{and}\qquad
  \nabla.
\]
They form the chain
\[
  \Delta<\cdots<\theta_{m+1}<\theta_m<\cdots<\theta_0<\nabla.
\]
\end{theorem}

\begin{proof}
First, every congruence class of a chain semilattice is convex: if
\(a\leq b\leq c\) and \(a\mathrel{\theta}c\), then
\[
  b=a\vee b\mathrel{\theta}c\vee b=c.
\]
If \(-\infty\mathrel{\theta}m\) for finite \(m\), multiplication by
\(-\infty\) gives
\[
  -\infty=(-\infty)\star(-\infty)
  \mathrel{\theta}m\star(-\infty)=m+1.
\]
Convexity identifies \(-\infty,0,\dots,m+1\), and successor compatibility
then collapses every larger integer.  Thus \(\theta=\nabla\).

We suppose that \(-\infty\) is a singleton class.  If all finite classes are
singletons, then \(\theta=\Delta\).  Otherwise convexity identifies an
adjacent pair; we let \(m\) be least with \(m\mathrel{\theta}m+1\).  Iterating
\(x\mapsto x\star x\) gives
\[
  m+j\mathrel{\theta}m+j+1
  \quad\text{for}\quad j\geq0,
\]
so \(T_m\) is one class, while minimality and convexity leave each smaller
integer as a singleton.  Hence \(\theta=\theta_m\).

Conversely, each \(\theta_m\) is compatible with maximum and product because
\(T_m\) is a final tail and any product with an input in \(T_m\) remains in
\(T_m\).  The displayed ordering follows from the nested tails.
\end{proof}

Each proper finite-index congruence collapses one final tail, so we name the
corresponding quotient before classifying the finite images.

\begin{definition}[Capped height algebra]\label{def:capped-height-algebra}
For \(m\in\NN\), the \defi{capped height algebra} is
\[
  \Hh_m\coloneqq\Hh/\theta_m.
\]
We write
\[
  \top\coloneqq[m]_{\theta_m}=T_m
\]
for its unique nonsingleton congruence class.  For \(m=0\), its elements are
\(-\infty\) and \(\top\).  For \(m\geq1\),
its elements are
\[
  -\infty,0,1,\dots,m-1,\top.
\]
Its product is obtained from \(\max\{a,b\}+1\) by collapsing every finite
value at least \(m\) to \(\top\).
\end{definition}

The congruence classification now gives the finite quotients immediately.

\begin{corollary}[Finite quotients]\label{cor:finite-quotients}
Every nontrivial finite quotient of \(\Hh\) is, up to isomorphism, a unique
capped height algebra \(\Hh_m\).
\end{corollary}

\begin{proof}
By \Cref{thm:congruences}, every proper finite-index congruence is a unique
\(\theta_m\).  The stated chain and product are the induced quotient
operations.
\end{proof}

Because the cap can be placed above either of two distinct finite heights,
these quotients also separate points.

\begin{corollary}[Residual finiteness]\label{cor:residual-finiteness}
The height algebra is residually finite, as is every finite power \(\Hh^n\).
\end{corollary}

\begin{proof}
If \(a<b\) are finite, then their images remain distinct in \(\Hh_b\), while
\(-\infty\) remains distinct from every finite integer in every proper capped
quotient.  Thus the capped quotients separate points of \(\Hh\).  Distinct
vectors in \(\Hh^n\) differ in some coordinate; projection to that coordinate,
followed by a separating capped quotient, separates them.
\end{proof}

The capped algebras are therefore the nontrivial finite homomorphic images of
\(\Hh\).  The first cases make the threshold collapse, and the resulting
product, explicit.

\begin{example}[The first capped quotients]\label{ex:capped-quotients}
The quotient \(\Hh_0\) has two elements, \(-\infty\) and \(\top\).  The
quotient \(\Hh_2\) has
\[
  -\infty<0<1<\top,
\]
and its products follow the height recursion until they reach the cap:
\[
  0\star0=1,
  \qquad
  0\star1=1\star1=\top,
  \qquad
  (-\infty)\star0=1,
  \qquad\text{and}\qquad
  (-\infty)\star1=\top.
\]
Thus \(\Hh_2\) retains the local height calculation and identifies all
heights at or above its threshold.
\end{example}

We now return from the intrinsic structure of \(\Hh\) to its tree
presentation and restrict to terms in which each variable occurs exactly once.

\section{The classical full-linear boundary}\label{sec:linear-spectra}

The normal-form theorem has a classical boundary in which no label repeats.
Two counting problems arise: one may keep the variables in a fixed order and
vary only the bracketing, or one may also permute the variables.  In the first
case, the profile is the left-to-right leaf-depth sequence of a plane full
binary tree; in the second, it is a labelled depth tuple.  The profile
classification reduces both counting problems to reconstruction and
enumeration of these depth data.

\begin{definition}[Full-linear terms and spectra]\label{def:full-linear-spectra}
A \defi{full-linear term over \(X\)} is a tree monomial whose leaves are
labelled bijectively by \(X\); equivalently, it contains each variable exactly
once and contains no silent leaves.  A binary operation is \defi{totally
nonassociative} if distinct bracketings of
\(x_1\star\cdots\star x_n\) induce distinct operations for every \(n\).
Following \cite[Introduction]{HuangLehtonen2023}, we let
\(s_n^{\mathrm a}(\star)\) denote the number of \(n\)-ary operations induced
by the bracketings of \(x_1\star\cdots\star x_n\) in that fixed order, and
we let \(s_n^{\mathrm{ac}}(\star)\) denote the number induced when the variables
may also be permuted.  These are the \defi{associative spectrum} and the
\defi{associative-commutative spectrum} of \(\star\), respectively.
\end{definition}

The ordered leaf-depth sequence is a standard encoding of a plane full binary
tree.  Huang and Lehtonen record its uniqueness property in
\cite[\S2.3]{HuangLehtonen2023}.  The short reconstruction below makes the
connection with height explicit.

\begin{lemma}[Ordered depth reconstruction]\label{lem:ordered-depth-reconstruction}
A plane full binary tree is uniquely determined by the left-to-right sequence
of its leaf depths.
\end{lemma}

\begin{proof}
We induct on the number \(n\) of leaves.  The case \(n=1\) is immediate.  If
\(n>1\) and the left root subtree contains the first \(k\) leaves, then its
leaf depths relative to that subtree are \(d_1-1,\dots,d_k-1\).  Kraft
equality in the left subtree gives
\[
  \sum_{i=1}^k2^{-(d_i-1)}=1,
\]
equivalently
\[
  \sum_{i=1}^k2^{-d_i}=\frac12.
\]
Before the last leaf of the left subtree the partial sum is strictly less than
\(1/2\), and after the first leaf of the right subtree it is strictly greater.
Thus the depth sequence determines \(k\) uniquely, and hence determines the two
root blocks.  Subtracting one from each depth in a block gives the ordered
depth sequence of that subtree, which induction reconstructs uniquely.
\end{proof}

Reconstruction converts the profile classification directly into the two
linear spectra.

\begin{theorem}[Linear spectra]\label{thm:linear-spectra}
For every \(n\geq1\), the following statements hold.
\begin{enumerate}[label=(\roman*)]
  \item
  \[
    s_n^{\mathrm a}(\star)=C_{n-1}
    =\frac1n\binom{2n-2}{n-1};
  \]
  in particular, \(\star\) is totally nonassociative;\label{item:catalan-count}
  \item \(s_n^{\mathrm{ac}}(\star)\) is the number of ordered tuples
  \((d_1,\dots,d_n)\in\NN^n\) satisfying
  \[
    \sum_{i=1}^n2^{-d_i}=1.
  \]
  Hence
  \[
    s_n^{\mathrm{ac}}(\star)=1,1,3,13,75,525,4347,\dots.
  \]\label{item:sequence-count}
\end{enumerate}
\end{theorem}

\begin{proof}
We fix the variable order.  A bracketing is a plane full binary tree whose leaves
are \(x_1,\dots,x_n\) from left to right.  Equal induced operations have equal
ordered depth sequences by \Cref{cor:tree-operation-equality}, and equal
ordered depth sequences determine the same plane tree by
\Cref{lem:ordered-depth-reconstruction}.  Thus all \(C_{n-1}\) bracketings
induce distinct operations.  The Catalan count is recalled in
\cite[Introduction]{HuangLehtonen2023}, proving \ref{item:catalan-count}.

We now allow variable permutations.  Every full-linear term yields a labelled
depth tuple satisfying Kraft equality.  Conversely, the equality case of
\Cref{thm:kraft-profile} realizes any such tuple.  Its prefix tree has no
internal vertex with a missing child: the corresponding binary cylinder would
otherwise have positive measure but contain no selected codeword, contradicting
Kraft equality.  Hence the realization has no silent leaves and is full-linear.
By \Cref{cor:tree-operation-equality}, its labelled depth tuple is a complete
invariant.  This proves \ref{item:sequence-count}, and the displayed values are the dyadic-composition
enumeration in \cite[Equation~(1.4)]{KrennWagner2016}.
\end{proof}

Thus fixed order retains the entire plane tree, whereas permutation retains
only the labelled depth tuple.  The second identification also transfers the
known asymptotics for dyadic compositions directly to the height product.

\begin{corollary}[Asymptotic associative-commutative spectrum]
\label{cor:ac-spectrum-asymptotic}
There are constants
\[
  \alpha=0.2963720490\dots,
  \qquad
  \gamma=1.1926743412\dots,
  \qquad\text{and}\qquad
  \rho=\frac{2}{3\gamma}<1
\]
such that
\[
  s_n^{\mathrm{ac}}(\star)
  =\alpha n!\,\gamma^{n-1}\bigl(1+O(\rho^n)\bigr).
\]
\end{corollary}

\begin{proof}
By \Cref{thm:linear-spectra}, the spectrum is the quantity denoted
\(\mathcal W_2(1,n)\) by Krenn and Wagner, so the formula is
\cite[Theorem~II]{KrennWagner2016}.
\end{proof}

The asymptotic agreement does not extend the full-linear analogy to repeated
variables.

\begin{remark}[Why the coincidence is only full-linear]
\label{rem:classical-means}
Huang and Lehtonen obtain the same associative-commutative spectrum for the
arithmetic, geometric, and harmonic means
\cite[Proposition~4.2.2]{HuangLehtonen2023}.  The coincidence is confined to
the full-linear layer, where each variable occurs once and all four operations
record its leaf depth.  Their nonlinear term theories are different.  For
example, the height terms
\[
  (x\star x)\star(x\star y)
  \qquad\text{and}\qquad
  (x\star y)\star(x\star y)
\]
have the same depth profile \((2,2)\), whereas the corresponding
arithmetic-mean terms give \(x\) the respective coefficients \(3/4\) and
\(1/2\).
\end{remark}

This boundary case therefore confirms that the profile quotient contains a
familiar classical layer, but it also marks where that analogy ends.  Repeated
labels and finite joins form the part of the theory developed here beyond
that classical layer.

\section{Further questions}
\label{sec:questions}

The preceding results give explicit normal forms and describe the coefficient
algebra, its language semantics, and its basic intrinsic structure.  Three
natural questions remain.  The first concerns the geometry of individual
profile fibres, the second concerns finite axiomatizability of the full
equational theory, and the third asks how much of the analysis persists for
higher branching arity.

\begin{question}[Reduced fibres of the depth-profile map]\label{ques:depth-fibres}
For \(n\geq1\) and \(d\in\Kraft_n\), how many silent-reduced labelled
nonplane rooted binary trees \(T\) satisfy \(d(T)=d\)?  How does this reduced
fibre size depend on the profile?  What asymptotics hold along sequences
\(d^{(k)}\in\Kraft_{n_k}\) as the number of finite coordinates, the largest
finite coordinate, or both tend to infinity?
\end{question}

The reduced fibre decomposes recursively at the root, so its enumeration may
be refined by counting unordered factorizations \(d=a\star b\) and the
representatives lying over each one.  A complementary problem is to determine
whether the trees in a reduced fibre are connected by finitely many
profile-preserving local moves.  Together with expansion and contraction of
all-silent subtrees, such moves would give a geometric presentation of the
kernel equivalence \(\sim_{\mathrm{prof}}\) of the depth-profile map.

The profile normal form decides whether any given identity holds, but it does
not determine whether all valid identities follow from finitely many of them.
This leads to a separate question about the equational theory of \(\Hh\).

\begin{question}[Finite equational basis]\label{ques:finite-basis}
Does the equational theory of \(\Hh\) in the signature
\((\vee,\star,-\infty)\) admit a finite basis, meaning a finite set
of identities from which every identity of \(\Hh\) follows?  If so, what is
such a basis?  If not, what obstruction prevents finite axiomatizability?
\end{question}

A third direction keeps the height recursion and changes only the branching
arity.

\begin{question}[Other branching arities]\label{ques:branching-arities}
We fix \(r\geq2\), and replace binary grafting by the recursion
\[
  h(T_1,\dots,T_r)=\max_i h(T_i)+1.
\]
What are the exact monomial-profile region, finite-join coefficient algebra,
and natural unit-and-zero completions?  After identifying full \(r\)-ary
depth tuples with the corresponding Kraft-equality compositions into powers
of \(r\), their full-linear asymptotics are given by
\cite[Theorem~III]{KrennWagner2016}.  What additional phenomena arise from
repeated labels and finite joins?
\end{question}

The binary theory developed here shows that weighted evaluation retains
exactly the greatest depth of each label.  The Kraft region governs single
trees, finite joins remove that constraint, and the resulting coefficient
algebra supports the language, completion, intrinsic, and full-linear
descriptions above.  The higher-arity question asks how much of this structure
is specific to binary branching.

\end{document}